\documentclass{aic}
\usepackage{enumerate}

\aicAUTHORdetails{%
  title = {A Rainbow Version of Mantel's Theorem}, 
  author = {Ron Aharoni, Matt DeVos, Sebasti\'an Gonz\'alez Hermosillo de la Maza, Amanda Montejano, and Robert \v{S}\'amal}, 
  plaintextauthor = {Ron Aharoni, Matt DeVos, Sebastian Gonzalez Hermosillo de la Maza, Amanda Montejano, Robert Samal}, 
  runningauthor = {Ron Aharoni, Matt DeVos, Sebasti\'an Gonz\'alez, Amanda Montejano, and Robert \v{S}\'amal}, 
  keywords = {Mantel's theorem, triangle, rainbow, extremal graph theory},
}   

\aicEDITORdetails{%
   year={2020},
   number={2},
   received={1 January 2019},   
   published={28 February 2020},  
   doi={10.19086/aic.12043},      
}   

\begin{document}

\begin{frontmatter}[classification=text]

\title{A rainbow version of Mantel's Theorem}

\author[ra]{Ron Aharoni \thanks{Supported in part by the United States--Israel Binational Science Foundation (BSF) grant no. 2006099, the Israel Science Foundation (ISF) grant no. 2023464 and the Discount Bank Chair at the Technion.}}
\author[md]{Matt DeVos\thanks{Supported by an NSERC Discovery Grant (Canada).}}
\author[sghm]{Sebasti\'an Gonz\'alez Hermosillo de la Maza}
\author[am]{Amanda Montejano \thanks{Supported by CONACyT 219827, DGAPA: PASPA and PAPIIT IN116519.}}
\author[rs]{Robert \v{S}\'amal \thanks{Supported by grant 19-21082S of the Czech Science Foundation. 
  This project has received funding from the European Union’s Horizon 2020 research and innovation programme under the Marie Skłodowska-Curie grant agreement No 823748.}}

\begin{abstract}
Mantel's Theorem asserts that a simple $n$ vertex graph with more than~$\frac{1}{4}n^2$ edges has a triangle (three mutually adjacent vertices).  Here we consider a
  rainbow variant of this problem.  We prove that whenever $G_1, G_2, G_3$ are simple graphs on a common set of $n$ vertices and 
  $|E(G_i)| > ( \frac{ 26 - 2 \sqrt{7} }{81})n^2 \approx 0.2557 n^2$ for $1 \le i \le 3$, then there exist distinct vertices $v_1,v_2,v_3$ so that 
  (working with the indices modulo 3) we have $v_i v_{i+1} \in E(G_i)$ for $1 \le i \le 3$. We provide an example to show this bound is best possible.
  This also answers a question of Diwan and Mubayi.  We include a new short proof of Mantel's Theorem we obtained as a byproduct.
\end{abstract}
\end{frontmatter}



\newtheorem{theorem}{Theorem}[section]
\newtheorem{lemma}[theorem]{Lemma}
\newtheorem{corollary}[theorem]{Corollary}
\newtheorem{proposition}[theorem]{Proposition}
\newtheorem{observation}[theorem]{Observation}
\newtheorem{definition}[theorem]{Definition}
\newtheorem{claim}{Claim}
\newtheorem{conjecture}[theorem]{Conjecture}
\newtheorem{problem}[theorem]{Problem}

\section{Introduction}

Throughout we shall assume that all graphs are simple.  For a positive integer $r$ we let $K_r$ denote a complete graph on $r$ vertices and we let $K_{r,r}$ denote a balanced complete bipartite graph with $r$ vertices in each part.  A \emph{triangle} in a graph $G$ is a subgraph isomorphic to $K_3$.  The starting point for this work is the following classical theorem, one of the first results in extremal graph theory.  

\begin{theorem}[Mantel \cite{Mantel}]\label{mant}
If $G$ is a graph on $n$ vertices with $|E(G)| > \frac{1}{4}n^2$, then $G$ contains a triangle.  
\end{theorem}

To see that this bound is best possible, observe that when $n$ is even, the complete bipartite graph $K_{\frac{n}{2},\frac{n}{2}}$ has $\frac{n^2}{4}$ edges but no triangle.  

In this article we consider a colourful variant of the above.  Let $G_1, G_2, G_3$ be three graphs on a common vertex set $V$ and think of each graph as having edges of a
distinct colour. Define a \emph{rainbow triangle} to be three vertices $v_1, v_2, v_3 \in V$ so that $v_i v_{i+1} \in E(G_i)$ (where the indices are treated modulo~3).
We will be interested in determining how many edges force the existence of a rainbow triangle.  
Is it true that if $|E(G_i)| > \frac{1}{4}n^2$ for $1 \le i \le 3$,  then there exists a rainbow triangle? 
By taking $G_1 = G_2 = G_3$ we return to the setting of Mantel's Theorem.  
In general, however, the answer to the former question is negative, as shown by the following construction.

Let $n$ be an integer and let $0 < t < \frac{1}{2}$ have the property that $tn$ is an integer.  Let $V$ be a set of $n$ vertices, and partition $V$ into $\{A,B,C\}$ where $|B| = |C| = tn$ and $|A| = (1 - 2t)n$.  Construct three graphs $G_1, G_2, G_3$ on $V$ as follows:  Let $G_1$ consist of a clique on $A$ plus a clique on $B$, let $G_2$ consist of a clique on $A$ plus a clique on $C$, and let $G_3$ consist of all edges except for those with both ends in $A$ (see Figure~\ref{fig:ABC}).  
\begin{figure}
  \centering
  \includegraphics[width=180pt]{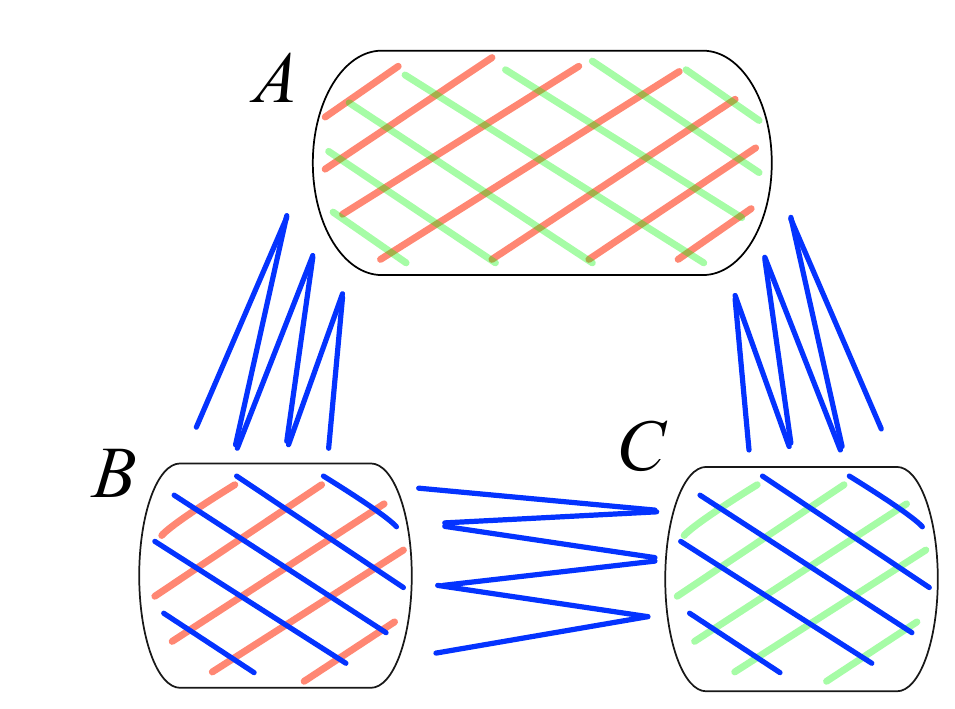}
  \caption{Graphs $G_1, G_2, G_3$ on $V=A\cup B\cup C$, with edges of $G_1$, $G_2$, $G_3$ depicted by colors red, green and blue respectively.}
  \label{fig:ABC}
\end{figure}
A simple check reveals that there is no rainbow triangle for this triple of graphs.  Furthermore 
$|E(G_1)| = |E(G_2)| = {n - 2tn \choose 2} + {tn \choose 2} = \frac{2-8t + 10t^2}{4}n^2 - \frac{1-t}{2}n$
while 
$|E(G_3)| = {n \choose 2} - {n - 2tn \choose 2}  = \frac{8t - 8t^2}{4} n^2 - tn$.

It is easy to verify that $2-8t + 10t^2>1$ and $8t - 8t^2>1$ whenever $t$ satisfies $\frac{1}{2}-\frac{1}{2\sqrt 2} < 0.147 < t < 0.155 < \frac{2-\sqrt{3/2}}{5}$. 
Thus, if we pick $0.147<t<0.155$, then for every sufficiently large $n$ (such that $tn$ is an integer) 
there are graphs $G_1, G_2, G_3$ on a common set of~$n$  vertices without a rainbow triangle that satisfy $|E(G_i)| > \frac{1}{4}  n^2$ for $1 \le i \le 3$. 
However, a slight increase in the number of edges forces the occurence of a rainbow triangle. 
Throughout the paper we fix the value $\tau = \frac{4 - \sqrt{7}}{9}$, so $\tau^2 \approx 0.0226$, and $\frac{1 + \tau^2}{4} =\frac{ 26 - 2 \sqrt{7} }{81} \approx 0.2557$.
\begin{theorem}
\label{main}
Let $G_1, G_2, G_3$ be graphs on a common set of $n$ vertices.  If $|E(G_i)| > \frac{1 + \tau^2}{4}n^2 $ for $1 \le i \le 3$, then there exists a rainbow triangle.  
\end{theorem}

Only after finishing the paper we have learned about work of A.~Diwan and D.~Mubayi~\cite{dhruv} (see also \url{https://faculty.math.illinois.edu/~west/regs/turancol.html}). 
They consider two-colored variants of Tur\'an's theorem, prove a couple of them and pose a problem about three-colored version of Mantel's theorem. Thus, the above 
theorem is an asymptotically tight solution to their problem. 

Theorem~\ref{main} is sharp in the sense that $\tau^2$ cannot be replaced by a smaller constant. To see this, note that $t = \tau$ is the unique solution to the quadratic
equation $2 - 8t + 10t^2 = 8t - 8t^2$ with $0 < t < \frac{1}{2}$.  For this number $\tau$ both sides of this quadratic equation are equal to $1 + \tau^2$.  It follows
by the construction  showed above (taking $n$ large and $t$ close to $\tau$) that for every $\epsilon > 0$ there exist simple graphs $G_1, G_2, G_3$ on a common set
of~$n$ vertices without a rainbow triangle that satisfy $|E(G_i)| > (\frac{1 + \tau^2}{4} - \epsilon) n^2$ for $1 \le i \le 3$.  

We were also able to use some of the ideas in our proof to obtain a new short proof of Mantel's Theorem. We include this proof at the beginning of the main section since it provides a nice example of a technique later used to prove Theorem \ref{main}.

Since $\tau^2$ is not rational, there does not exist a graph $G$ with $|V(G)| = n$ and $|E(G)| = \frac{1 + \tau^2}{4}n^2$, and thus there is no finite tight
example for our problem. This inconvenience is removed in the setting of graph limits and graphons: a growing sequence constructed as in the previous
paragraph would converge to three graphons, each with density $\frac{1 + \tau^2}{2}$ and without a rainbow triangle. In this setting, Razborov's flag algebra machinery may give an alternative proof of our result,
and be useful in extending it. Indeed, a flag-algebra proof has already been obtained (independently from us) by E.~Culver, B.~Lidick\'y, F.~Pfender, and J.~Volec~\cite{flags-manu}. 
Their proof even gives a precise characterization of all extremal configurations with sufficiently large number of vertices. 
To further explore this area of ``rainbow extremal graph theory'' each approach has its pluses and minuses, ours has the advantage of being verifiable by hand. We suggest some potential interesting directions to proceed with the following problems.

\begin{problem}
For what real numbers $\alpha_1, \alpha_2, \alpha_3 > 0$ is it true that every triple of graphs $G_1, G_2, G_3$ satisfying $|E(G_i)| > \alpha_i n^2$ 
must have a rainbow triangle?
\end{problem}

Tur\'an's Theorem generalizes Mantel's Theorem by proving that for every integer $r \ge 2$, a simple $n$ vertex graph with more than $(1 - \frac{1}{r-1}) \frac{n^2}{2}$
edges has a $K_r$ subgraph. Analogously, one may consider the following.

\begin{problem}
For every positive integer $r$, what is the smallest real number $\delta_r$ so that whenever $G_1, \ldots, G_{ {r \choose 2} }$ are graphs on a common set of $n$ vertices
with $|E(G_i)| \ge \delta_r n^2$ for every $1 \le i \le {r \choose 2}$ there exists a rainbow $K_r$, i.e., a set of $r$ vertices and one edge from each $G_i$ that
together form a clique on this set of vertices.
\end{problem}

We can also consider this problem for the number of graphs (colours) being different from~$\binom r2$, with an appropriately modified notion of ``rainbow''.  
For $r=3$ and more than three graphs the answer is $\delta_r = 1/4$ with the extremal configuration beeing all graphs identical complete biparite 
(Theorem 1.2 of~\cite{KSSV}). 
When the number of graphs is less than~$\binom r2$, one can  study the existence of other colour patterns. For $r=3$ and
two graphs a problem with such a flavor was considered in ~\cite{DMM}.  

We will finish this section by a sample of other results and conjectures that can be described as rainbow. 
Perhaps historically first is a result of B\'ar\'any~\cite{Barany} in combinatorial geometry. 
He obtained a rainbow (also termed colourful) version of Carath\'eodory's theorem; see also~\cite{GKblog}.

More recent, and closer to our present topic, is the study of rainbow Erd\H{o}s-Ko-Rado theorems. 
Let us use $f(n,r,k)$ to denote the EKR number -- the smallest $m$ such that every $r$-uniform hypergraph with $n$~vertices and $m$~edges 
has a matching of size~$k$. (Recall that the classical Erd\H{o}s-Ko-Rado theorem states that $f(n,r,2) = \binom{n-1}{r-1}$ whenever $n \ge 2r$.) 
Aharoni and~Howard~\cite{rainbowEKR} conjecture the following rainbow version: 

\begin{conjecture}[Aharoni, Howard] 
  Let $H_1$, \dots, $H_k$ be $r$-uniform hypergraphs on the same set of~$n$~vertices, each having $f(n,r,k)$ hyperedges. 
  Then there is a rainbow matching: a matching $\{e_1, \dots, e_k\}$ such that $e_i \in E(H_i)$ for $i=1, \dots, k$. 
\end{conjecture}

In~\cite{rainbowEKR} this conjecture is discussed for hypergraphs that are balanced $r$-partite and it is proved there with this restriction for $r=3$. 
We finish with a conjecture motivated by Dirac's condition for hamiltonicity. 

\begin{conjecture}[Aharoni] 
  Given graphs $G_1$, \dots, $G_n$ on the same vertex set of size~$n$, each having minimum degrees at least $n/2$, there exists a rainbow Hamilton cycle: 
  a cycle with edge-set $\{e_1, \dots, e_n\}$ such that $e_i \in E(G_i)$ for $i=1, \dots, n$.
\end{conjecture}

\section{Proof}

A key idea in the proof of Theorem \ref{main} will be to analyze the structure of the subgraphs induced by pairs of edges in a matching of a graph. We can use a similar (but simpler) approach to obtain a short proof of Mantel's Theorem.

\begin{lemma}\label{count}
	Let $G$ be a graph and $P$ be the set of pairs of distinct vertices $\left\lbrace x,y\right\rbrace \subseteq V(G)$ such that $N(x) \cap N(y) \neq \emptyset$. If $M$ is a maximal matching in $G$, then~$\left| P\right| \geq  \left| E(G) \right| - \left| M \right|$. 
\end{lemma}

\begin{proof}
	
	Let $M = \left\lbrace e_1, e_2, \dots, e_k \right\rbrace$ be a maximal matching of $G$. Since $M$ is maximal, we know that every edge $e \in E(G)$ has at least one endpoint in common with an edge of $M$. For $e \in E(G) \setminus M$, let $s(e)$ be the smallest integer such that $e \cap e_{s(e)} \neq \emptyset$, and take $f(e) = e \triangle e_{s(e)}$. It is easy to see that $f:E(G) \setminus M \to P$ is an injective function, and the result follows.
\end{proof}

\begin{proof}[Proof of Theorem \ref{mant}]
	Let $G$ be a triangle-free graph and $M$ a maximum matching of $G$. Since $G$ has no triangles then $ \left| P \right| + \left| E(G) \right| \leq \binom{n}{2}$, and by Lemma \ref{count} we have $\left| E(G)\right| - \frac{1}{2}n \leq  \left| E(G) \right| - \left| M \right| \leq \left| P\right|$. By combining these inequalities, we get $2\left| E(G) \right| \leq \binom{n}{2} + \frac{1}{2}n$, and so $\left| E(G) \right| \leq \frac{1}{4}n^2$.
\end{proof}

There are a great many proofs of Mantel's Theorem and we borrow ideas from a few. In particular, the following lemma we require is a variant of an ``entropy minimizing''
proof (\cite{Erdos-Turan}, see also \cite{Aigner-Turan}). 
For a graph $G$ and a set $X \subseteq V(G)$ we let $G[X]$ denote the subgraph of $G$ induced on $X$ and we let $e_G(X) = |E(G[X])|$.  If $Y \subseteq V(G)$
is disjoint from~$X$ we let $e_G(X,Y) = | \{ xy \in E(G) \mid \mbox{$x \in X$ and $y \in Y$} \}|$.  As usual, if the graph $G$ is clear from context, we drop
the subscript $G$.

\begin{lemma}
\label{bipman}
Let $G$ be a graph and let $\{Z_0, Z_1\}$ be a partition of $V(G)$.  If for $i=0,1$, every $z \in Z_i$ has the property that $N(z) \cap Z_{1 - i}$ is a clique, then 
\[ e(Z_0,Z_1) \le e(Z_0) + e(Z_1) + \tfrac{1}{2} \left( |Z_0| + |Z_1| \right). \]
\end{lemma}

\begin{proof}
We say that two vertices $z,z' \in Z_i$ ($i=0,1$ is the same for both $z$, $z'$) are \emph{twins} if they have the same closed neighbourhood, $N[z] = N[z']$.
Observe that being twins is an equivalence relation.  
Now we choose a graph $G$ so that
\begin{enumerate}[(1)]
  \item \label{opt1} $e(Z_0,Z_1) - e(Z_0) - e(Z_1)$ is maximum; and 
  \item \label{opt2} the total number of pairs of vertices that are twins is maximum (subj. to (\ref{opt1})).
    Recall that if a pair of vertices are twins, then they both are in~$Z_0$ or both in~$Z_1$. 
\end{enumerate}  
Observe that it is sufficient to verify the desired bound for this~$G$. 
Now we fix $i = 0,1$ and let $w',w'' \in Z_i$ be adjacent.
Consider the graph $G'$ ($G''$) obtained from $G$ by deleting $w'$ ($w''$) and adding a new vertex and making it a twin of $w''$ ($w'$).
It is immediate from this construction that both $G'$ and $G''$ satisfy the condition that 
$N(z) \cap Z_{1-j}$ is a clique for every $j = 0,1$ and $z \in Z_j$.  If one of $G'$ or $G''$ is superior to the
original graph $G$ for the first optimization criterion~(\ref{opt1}) we have a contradiction with the choice of~$G$.  It follows that all three graphs $G$, $G'$, and $G''$ are tied relative to
this criterion.  If $w'$ and $w''$ are not twins, then one of $G'$ or $G''$ is superior to $G$ relative to the second optimization criterion~(\ref{opt2}).  
To see this, observe that if $x, y \in Z_{1-i}$ are twins in~$G$, then they are also twins in~$G'$ and in~$G''$. 
If the twin-equivalence class of~$w'$ ($w''$) has $a'$ ($a''$ elements), then $G'$ looses $a'-1$ twin-pairs and gains $a''$ new ones; thus if 
$a'' \ge a'$, then $G'$ is superior to~$G$, a contradiction to the choice of~$G$. (If $a' > a''$, we use $G''$.) 
It follows that $w'$ and $w''$ must be twins. 
 
As being twins is an equivalence relation, we conclude that the graph~$G$ is a disjoint union of complete graphs.  
Consider a component $H$ of $G$ with $|V(H) \cap Z_0| = \ell$ and $|V(H) \cap Z_1| = m$.  In this case the sets $C = E(H) \cap E(Z_0,Z_1)$ and $D = E(H) \setminus C$ satisfy
  $$
   |D| - |C| = {\ell \choose 2} + {m \choose 2} - \ell m = \tfrac{1}{2}(\ell - m)^2 - \tfrac{\ell + m}{2} \ge -\tfrac{\ell + m}{2} 
  $$
and the lemma follows by summing these inequalities over each component.	
\end{proof}

Any counterexample to Theorem~\ref{main} would immediately imply the existence of large counterexamples by way of ``blowing up'' vertices.  More precisely, suppose that
$G_1, G_2, G_3$ contradict the theorem, and let $k$ be a positive integer.  Now replace every vertex $v$ by a set $X_v$ consisting of $k$ isolated vertices, and for each
graph $G_i$, replace every edge $uv$ by all possible edges between the sets $X_u$ and $X_v$.  This operation magnifies the number of vertices by a factor of $k$ and the
number of edges in each $G_i$ by a factor of $k^2$, and thus yields another counterexample.  
Moreover, if $\min_{1 \le i \le 3} \frac{|E(G_i)|}{n^2} - \frac{1+\tau^2}{4}  = \epsilon$ then this property will also be preserved.  So, the resulting graph 
on~$kn$ vertices will exceed the bound by $\epsilon k^2n^2$ edges ($\epsilon$ is positive, as $\tau^2$ is irrational). 
The condition in Theorem \ref{main} implies $|E(G_i)| + |E(G_j)| \ge \tfrac{1 + \tau^2}{2} n^2$ for all $1 \le i<j\le 3$. 
Hence, by the above observation, to prove Theorem~\ref{main} it suffices to establish the following result.

\begin{lemma}\label{main2}
Let $G_1, G_2, G_3$ be graphs on a common set $V$ of $n \ge 1$ vertices.  If 
\[ |E(G_i)| + |E(G_j)| \ge \tfrac{1 + \tau^2}{2} n^2 + \tfrac{3}{2}n \]
holds for every $1 \le i < j \le 3$, then there exists a rainbow triangle.
\end{lemma}

The statement of the above lemma replaces a bound on the number of edges in each graph $G_i$ by a bound on the sum of the number of edges in any two such graphs.
This adjustment will allow us to forbid certain types of induced subgraphs of a possible minimal counterexample, as demonstrated by Lemma \ref{no3pm} below.

To proceed, we need some further notation. For the remainder of this article we
will be focusing on the proof of Lemma~\ref{main2}, so we will always have three
graphs $G_1, G_2, G_3$ on a common set of vertices $V$.  Abbreviating our usual
notation, if $X \subseteq V$ and $1 \le i \le 3$ we will let $e_i( X) =
e_{G_i}(X)$ and we define $e(X) = e_1(X) + e_2(X) + e_3(X)$.  Similarly, if $Y
\subseteq V$ is disjoint from $X$, then we let $e_i(X,Y) = e_{G_i}(X,Y)$ and we
define $e(X,Y) = e_1(X,Y) + e_2(X,Y) + e_3(X,Y)$. We also let $E_i = E(G_i)$. 

\begin{lemma}\label{no3pm}
A counterexample to Lemma~\ref{main2} with $n$ minimum does not contain a nonempty proper subset of vertices $X$ for which $G_i[X]$ has a perfect matching for all $1 \le i \le 3$.
\end{lemma}

\begin{proof}
  Let $G_1$, $G_2$, $G_3$ be a minimal counterexample to Lemma~\ref{main2}. Suppose (for a contradiction) that there is a set of vertices $X$, such that every colour induces a graph  with a perfect matching on $X$.  Let $|X| = \ell$ and let $M$ be a perfect matching in the graph
  $G_3[X]$.  If $xx' \in M$ and $y \in V \setminus \{x,x'\}$, then $e_1(y, \{x,x' \} ) + e_2(y, \{x,x'\}) \le 2$ (otherwise there would be a rainbow triangle).  
  Summing this over all edges of $M$ and $y \in V \setminus X$ gives us $e_1(X, V \setminus X) + e_2(X, V \setminus X) \le \ell(n- \ell)$.  

If $uu', vv' \in M$ and $uv \in E_1\cap E_2$, then  $uv', vu' \not \in E_1 \cup E_2$. 
This implies that $e_1(\{u,u'\}, \{v,v'\}) + e_2(\{u,u'\}, \{v,v'\}) \le 4$ and thus, 
on average, an edge in $E(X)\setminus M$ contributes at most $1$
to $e_1(X) + e_2(X)$; obviously an edge of $M$ contributes at most~$2$. 
Hence $e_1(X) + e_2(X) \le \frac{\ell}{2} + {\ell \choose 2} = \frac{\ell^2}{2}$.
Therefore
\begin{align*}
|E(G_1 - X)| + |E(G_2-X)| 
  &\ge |E(G_1)| + |E(G_2)| - \ell (n-\ell) - \tfrac{\ell^2}{2} \\
	&\ge \tfrac{1+\tau^2}{2}n^2 + \tfrac{3}{2}n  - \ell n + \tfrac{\ell^2}{2} \\
	&\ge \tfrac{1+\tau^2}{2} (n-\ell)^2 +  \tfrac{3}{2}(n - \ell).
\end{align*} 
It follows from the same argument applied to the other two pairs of colours that the graphs $G_1-X$, $G_2-X$, and $G_3-X$ form a smaller counterexample, contradicting minimality.
\end{proof}

\begin{figure}[ht]
  \centering
  \includegraphics[width=310pt]{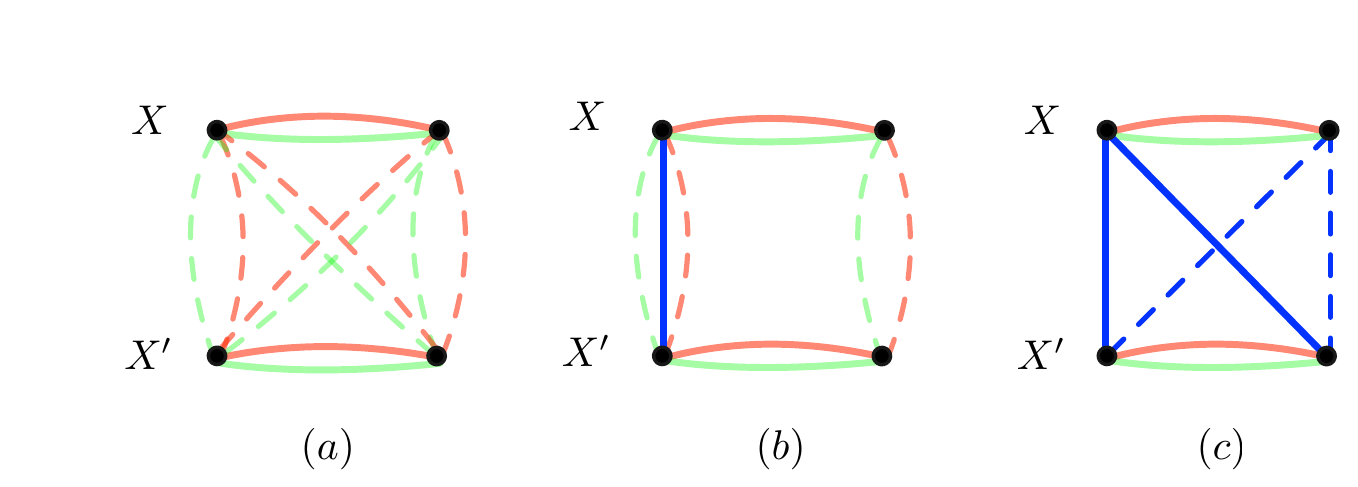}
  \caption{The three situations described in the first instance of Observation \ref{digonobs}. Colors $i, j$ and $k$ are depicted by green, red and blue respectively.
  If an edge is not depicted it means that the edge is not present, if it is dashed it indicates that it may be present.}
  \label{fig:Obs1}
\end{figure}

\begin{figure}[ht]
  \centering
  \includegraphics[width=260pt]{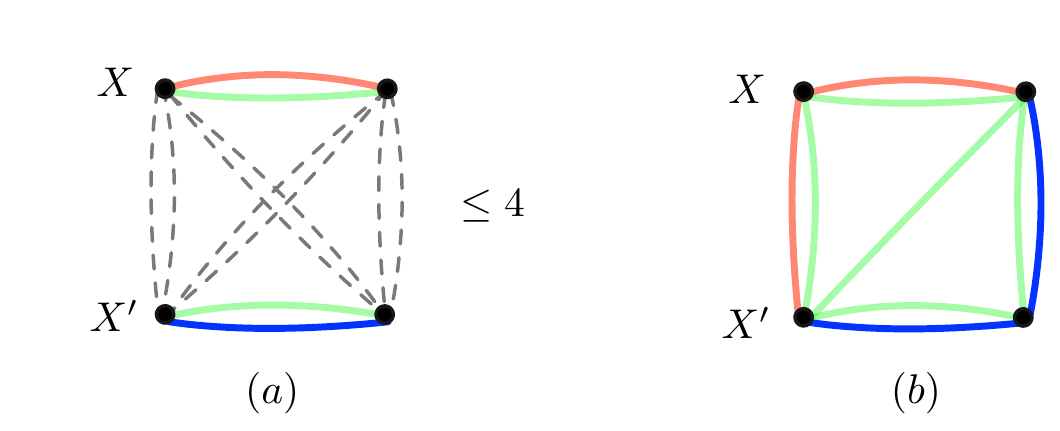}
  \caption{The two situations described in the second instance of Observation \ref{digonobs}. Colors $i$, $j$ and $k$ are depicted by green, red and blue respectively.
    If an edge is not depicted it means that the edge is not present, if it is dashed  (gray) it indicates that it may be present in any color (as soon as no rainbow triangle appears).}
  \label{fig:Obs2}
\end{figure}

\begin{observation}
\label{digonobs}
Let $G_1, G_2, G_3$ be a counterexample to Lemma~\ref{main2} for which $n$ is minimum and let $X, X' \subseteq V$ be disjoint sets with $|X| = |X'| = 2$. 
Suppose also that $n \ge 5$. 
If $e(X) = e(X') = 2$ and $\{i,j,k\} = \{1,2,3\}$ then we have:
\begin{enumerate}
\item If $e_i(X) = e_j(X) = e_i(X') = e_j(X') = 1$, one of the following holds (see Figure~\ref{fig:Obs1}):
\begin{enumerate}
\item $e_{k}(X,X') = 0$, or
\item $e_{k}(X,X') = 1$ and $e_{i}(X,X'), e_{j}(X,X') \le 2$, or
\item $e_{k}(X,X') = 2$ and $e_{i}(X,X') = e_{j}(X,X') = 0$.
\end{enumerate}
\item If $e_i(X) = e_j(X) = e_i(X') = e_k(X') = 1$, one of the following holds (see Figure~\ref{fig:Obs2}):
\begin{enumerate}
\item $e(X,X') \le 4$, or
\item $e(X,X') = 5$ where $e_{i}(X,X') = 3$ and $e_{j}(X,X') = e_{k}(X,X') = 1$.
\end{enumerate}
\end{enumerate}
\end{observation}

\begin{proof}
The first part follows from the observation that the graph induced on $X \cup X'$
  cannot have two nonadjacent edges of colour $k$ (otherwise Lemma~\ref{no3pm} would be violated) and a straightforward case analysis (leaning on the assumption that there is no rainbow triangle).  

To show the second part, recall that  Lemma~\ref{no3pm} implies there is no pair of vertices adjacent in
all three colors. We make a similar case analysis, obtaining that the  only configuration with  $e(X,X') = 5$ is the one depicted in Figure \ref{fig:Obs2} $(b)$.
\end{proof}

\begin{proof}[Proof of Lemma \ref{main2}]
Suppose (for a contradiction) that Lemma~\ref{main2} is false, and choose a counterexample 
$G_1, G_2, G_3$ with common vertex set $V$ so that $n = |V|$ is minimum.  
  It follows that $n > 5$ (otherwise the given bound on~$|E(G_i)| + |E(G_j)|$ is greater than $2\binom n2$). 
  Recall that by Lemma~\ref{no3pm} 
there does not exist a pair of vertices adjacent in all three graphs $G_1, G_2, G_3$.  Say that a set $X \subseteq V$ with $|X|=2$ is a \emph{digon} if $e(X) = 2$.  Now, choose a maximum sized collection $M$ of pairwise disjoint digons.  For $1 \le i < j \le 3$ define $M_{i,j}$ to be the subset of $M$ consisting of those digons $X$ so that $e_i(X) = e_j(X) = 1$.  For every $1 \le i < j \le 3$ let $X_{i,j}$ be the union of the digons in $M_{i,j}$ and let $D = V \setminus (X_{1,2} \cup X_{1,3} \cup X_{2,3})$.  
 
\begin{claim} 
The set $D$ satisfies the following.
\begin{enumerate}
  \item If $x,x' \in D$ then $e(x,x') \le 1$. 
  \item If $X \in M_{i,j}$ and $y \in D$ satisfy $e(X,y) \ge 3$, then $e_k(X,y) = 0$ (with $\{i,j,k\} = \{1,2,3\}$). 
  \item For every $X \in M$ there is at most one vertex $y \in D$ for which $e(X,y) = 4$.  
\end{enumerate}
\end{claim}

\begin{proof}[Proof of Claim~1]
The first part follows from the maximality of $M$.
For the second part, if $X=\{x,x'\}$ and, say, $xy \in E_k$ then $x'y \notin E_i \cup E_j$, thus $e(X,y) \le 2$.  
For the last part, assume again $X=\{x,x'\} \in M_{i,j}$. For contradiction, suppose there are distinct $y,y' \in D$ 
such that $e(X,y) = e(X,y') = 4$. As no edge is contained in all three subgraphs, both $\{x,y\}$ and $\{x',y'\}$ 
are digons, so we may use them in~$M$ instead of $\{x,x'\}$, contradicting the maximality of $M$.
\end{proof}

\bigskip

  The plan for the rest of the proof is to use our understanding of the structure of minimal counterexample 
  (as given by Observation~\ref{digonobs} and by Claim~1) to 
  derive several inequalities for $|X_{1,2}|$, $|X_{2,3}|$, $|X_{1,3}|$ and~$|D|$. These inequalities will appear 
  as (1)--(3) below. As we will show they have no solution, we will reach our desired contradiction. 

In order to simplify calculations we now replace the graphs $G_i$ by graphs having simpler structure (and possibly several 
rainbow triangles). 
First off, for every $X \in M$, if there exists $y \in D$ with $e(X,y) = 4$, then delete one edge between $X$ and $y$.  (Note that
by the above claim this removes at most $\frac{n}{2}$ edges in total.)    
Next suppose that $\{i,j,k\} = \{1,2,3\}$ and $X,X' \in M_{i,j}$ satisfy $e_{k}(X,X') > 0$.  
In this case we delete all edges between $X$ and $X'$ (in all three graphs) and then add
back three edges between $X$ and $X'$ in $G_i$ and three such edges in $G_j$.  
Note that for this operation the first part of Observation~\ref{digonobs} implies that the sum of the edge-count 
for any two of the graphs does not decrease.  
Let $G_1', G_2', G_3'$ be the graphs resulting from applying these operations whenever possible. By the above, 
\[ 
  \min_{1 \le i < j \le 3} |E(G'_i)| + |E(G'_j)| \ge \tfrac{1+\tau^2}{2}n^2 + n. 
\]
Note that the sets $M$, $M_{i,j}$, $X_{i,j}$ and $D$ do not change by going to $G'_1$, $G'_2$, $G'_3$. 
The graphs $G'_1, G'_2,G'_3$ may have a rainbow triangle. However, each such triangle involves an edge between 
two digons $X$, $X'$ in the same set $M_{i,j}$ for which $e_{G'_i}(X,X') = e_{G'_j}(X,X') = 3$. 

Before making our next modification, we pause to construct an auxiliary graph.  For every $1 \le i \le 3$ let $\{i,j,k\} = \{1,2,3\}$ and construct a simple graph $H$ with vertex set $M_{i,j} \cup M_{i,k}$ by the following rules:
\begin{itemize}
  \item If $X,X' \in M_{i,j}$ are distinct, we add an edge between them if $e_{G_j'}(X,X') \le 3$. 
    We do the same with $k$ in place of~$j$. 
  \item If $X \in M_{i,j}$ and $X' \in M_{i,k}$, we add an edge between them if $e'(X,X') = 5$
    (using $e' = e_{G_1'} + e_{G_2'} + e_{G_3'}$). 
\end{itemize}
\begin{claim} The graph $H$ satisfies the hypothesis of Lemma~\ref{bipman}
  with  $Z_0 = M_{i,j}$ and~$Z_1 = M_{i,k}$.
\end{claim}
\begin{proof}[Proof of Claim~2]
For the sake of contradiction assume that there are digons $X,X' \in M_{i,k}$ and $Y \in M_{i,j}$ such that 
both $YX$ and $YX'$ are edges of~$H$, but $XX'$ is not. 
It follows that $e_{G'_k}(X,X') = 4$, so the edges between $X$ and~$X'$ have not been modified when constructing graphs $G'_1$, $G'_2$, $G'_3$. 
Moreover, as $e'(X,Y) = e'(X',Y) = 5$, the second part of Observation \ref{digonobs} describes precisely the structure of edges 
between $X \cup X'$ and $Y$, and we find a rainbow triangle (not only in~$G'$, but also in~$G$). 
\end{proof}

By Lemma~\ref{bipman} we have $e_H(Z_0) + e_H(Z_1)\ge e_H(Z_0,Z_1)-(|M_{i,j}|+|M_{i,k}|)/2$. 
Note that by the definition of $H$, $e_H(Z_0)$ is at most the number of missing edges of colour $j$ in $X_{i,j}$, 
and $ e_H(Z_1)$ is at most the number of missing edges  of colour~$k$ in~$X_{i,k}$.
Also, $e_H(Z_0,Z_1)$ is the number of pairs $(X,X')$ where $X \in M_{i,j}$ and $X' \in M_{i,k}$ that satisfy $e(X,X') = 5$.

Based on this observation, we now construct a new triple of graphs. 
For each $i = 0, 1, 2$ and for $j<k$ such that $\{i,j,k\}=\{0,1,2\}$ we modify 
the induced subgraphs of $G_1'$, $G_2'$, $G_3'$ on $X_{i,j} \cup X_{i,k}$: 

{
  \narrower
For every $X \in M_{i,j}$ and $X' \in M_{i,k}$ we modify the graph between $X$ and $X'$ as follows:
If $e(X,X') = 5$ (so, by Lemma~\ref{bipman}, $e_i(X,X') = 3$ and $e_j(X,X') = 1 = e_k(X,X')$) 
then delete the edges between $X$ and $X'$ of colours $j$ and $k$ and then add back one new edge of colour $j$ or $k$ (to be chosen later) so that the new edge is
not  parallel to any of the three edges of colour~$i$.  Otherwise, if $e(X,X') \le 4$ we rearrange the edges between $X$ and $X'$ so that every $x \in X$ and $x' \in X'$ satisfy $e(x,x') \le 1$.
Next, we add to $X_{i,j}$ all missing edges of colour~$j$ and to $X_{i,k}$ all missing edges of colour~$k$. 
\def\loss{\mathop{\mathrm{loss}}}
We let $\loss(j,i)$ denote the decrease of the number of edges of~$G_j'$ (thus $\loss(j,i)$ is negative, if the number of edges of~$G_j'$ has increased). 
Similarly for $k$ in place of~$j$. 
It follows from Lemma~\ref{bipman} and the above discussion that $\loss(j,i) + \loss(k,i) \le (|M_{i,j}| + |M_{i,k}|)/2 \le n/4$. 
Moreover, we can choose whether to decrease the number of edges of~$G_j'$ or $G_k'$. Thus, we may ensure that 
$\loss(j,i)<0$ only if $\loss(k,i)\le 0$ (and vice versa). Consequently, we have 
$$
  \loss(j,i) \le n/4  \quad \mbox{and} \quad \loss(k,i) \le n/4. 
$$

}

Summing up, we may arrange the modification process so that each colour class of edges decreases in size by at most $\frac{n}{2}$.  
So, if we let $G''_1, G''_2, G''_3$ be the graphs resulting from our operation, we have:
\[
   \min_{1 \le i < j \le 3} |E(G''_i)| + |E(G''_j)| >  \tfrac{1+\tau^2}{2}n^2.
\]

To complete the proof, we will now show that the above density condition is incompatible with the structure of the graphs $G''_i$.  
Let $G''=\bigcup_{i\le 3}G''_i$. Below we use the notation $e(S)$, $e(S,S')$, $e_i(S)$ for the parameters in the graph $G''$. 
The construction of the graphs $G''_i$ implies:

\begin{claim}
\begin{enumerate}
  \item The subgraph of $G''$ induced on $X_{i,j}$ is complete in colours $i$ and $j$ and empty in the remaining colour.
  \item If $x \in X_{i,j}$ and $x' \in X_{i,k}$ where $j \neq k$, then $e(x,x') \le 1$.
  \item If $x,x' \in D$, then $e(x,x') \le 1$.
  \item \label{denstod} If $X \in M_{i,j}$ and $y \in D$ then $e(X,y) \le 3$ and if $e(X,y) = 3$, then all edges between $X$ and $y$ have colour $i$ or colour $j$.
\end{enumerate}
\end{claim}

\def\dd{\mathop{\mathtt{d}}}

Let $a$, $b$, $c$, $d$ be such that $an = |X_{1,2}|$, $bn = |X_{1,3}|$, $cn= |X_{2,3}|$, and $dn = |D|$.  
We shall assume (without loss of generality) that $a \ge b \ge c$, and we note that $a+b+c+d = 1$.
Next, we will apply our density bounds to get some inequalities relating $a$, $b$, $c$, and $d$.  
For the purposes of these calculations, it is convenient to introduce a density function.  
For any graph $H$ we define $\dd(H) = \frac{ 2 |E(H)| }{ |V(H)|^2 }$.  Note that with this terminology we have

\[ \min_{1 \le i < j \le 3} \dd(G''_i) + \dd(G''_j) > 1 + \tau^2. \]

First we consider just colours 2 and 3.  Note that if $x,y \in V$ are adjacent in both $G_2''$ and $G_3''$, 
then either $x,y \in X_{2,3}$ or one of these vertices is in $X_{2,3}$ and the other is in $D$.
Furthermore, in this last case, if say $y \in D$ and $x \in X_{2,3}$ has $X = \{x,x'\} \in M_{2,3}$, then $e_2(X, y) + e_3(X,y) \le 3$.  
It follows from this that $|E(G_2'')| + |E(G_3'')| \le {n \choose 2} + {cn \choose 2} + \frac{1}{2} cdn^2 \le \frac{n^2}{2} + c^2 \frac{n^2}2 + cd \frac{n^2}{2}$.  
Multiplying this equation through by $\frac{2}{n^2}$ then gives the useful bound
\begin{equation}
\label{011}
  c^2 + cd \ge \dd(G_2'') + \dd(G_3'') - 1 \ge \tau^2. 
\end{equation}
Next, we will count edges of~$G_1''$ twice, edges of~$G_2''$ twice, and edges of~$G_3''$ three times. 
An edge within~$X_{1,2}$ is counted four times in total, edge within $X_{1,3}$ or within $X_{2,3}$ five times in total. 
Edge between $X_{1,2}$ and $X_{2,3}$ (etc.) at most three times. 
Finally, for $y \in D$ and $X \in M_{2,3}$, we count the two edges between $y$ and $X$ at most 
$2+3+3$ times, thus we have at most $4 |D| |X_{2,3}|$ edges here. The same count applies for $M_{1,3}$ in place of~$M_{2,3}$; 
for $M_{1,2}$ we get at most $3 |D| |X_{1,2}|$. This implies 
$$
  2 |E(G_1'')| + 2|E(G_2'')| + 3 |E(G_3'')| -3 {n \choose 2} \le {an \choose 2} + 2 {bn \choose 2} + 2 {cn \choose 2} + bdn^2 + cdn^2, 
$$
giving us the bound
\begin{equation*}
  a^2 + 2b^2 + 2c^2 +2bd + 2cd \ge 2 \dd(G_1'') + 2 \dd(G_2'') + 3 \dd(G_3'') - 3. 
\end{equation*}
We can express the right hand side as 
$\tfrac12 (\dd(G_1'')+\dd(G_2'')) +\tfrac32 (\dd(G_1'')+\dd(G_3'')) +\tfrac32 (\dd(G_2'')+\dd(G_3'')) -3$
and use the lower bound for each of the sum of two densities, yielding 
\begin{equation}
\label{223}
  a^2 + 2b^2 + 2c^2 +2bd + 2cd \ge \tfrac{1}{2} + \tfrac{7}{2}\tau^2. 
\end{equation}

Finally, $|E(G_1'')| + |E(G_2'')| + |E(G_3'')| - {n \choose 2} \le {an \choose 2} + {bn \choose 2} + {cn \choose 2} + \frac{1}{2}(a+b+c)dn^2$ gives us the inequality
(strict one, as $\tau^2$ is irrational) 
\begin{equation}
\label{111}
  a^2 + b^2 + c^2 + d(a+b+c) > \tfrac{1}{2} + \tfrac{3}{2}\tau^2.
\end{equation}

We claim that there do not exist nonnegative real numbers $a,b,c,d$ with $a \ge b \ge c$ and $a+b+c+d = 1$ satisfying the inequalities (\ref{011}), (\ref{223}), and (\ref{111}).  To prove this, first note that inequality (\ref{011}) (and the quadratic formula) imply
\begin{equation}
\label{easycl} 
  c \ge \frac{-d + \sqrt{d^2 + 4\tau^2}}{2}.
\end{equation}

Now $b \ge c$ and inequality (\ref{easycl}) imply $b + c + d \ge \sqrt{d^2 + 4\tau^2}$.  This gives us the following useful upper bound on $a$
\begin{equation}
\label{aupper}
  a \le 1 - \sqrt{d^2 + 4\tau^2} \le 1 - 2\tau.  
\end{equation}
To get a lower bound on $a$, observe that $a \ge b \ge c$ and inequality (\ref{111}) give us
$\frac{1}{2} + \frac{3}{2}\tau^2 < a^2 + b^2 + c^2 + d(1-d) \le 3a^2 + \frac{1}{4}$.  It follows that $a \ge \sqrt{ \frac{1}{12} + \frac{1}{2}\tau^2 } \ge 2\tau$.  Combining this lower bound on $a$ with the upper bound (\ref{aupper}) gives the following useful inequality
\[ a^2 + (1-a)^2 \le 1 - 4\tau + 8\tau^2. \]
The above bound together with (\ref{223}) implies
\[\tfrac{1}{2} + \tfrac{7}{2}\tau^2 \le a^2 + (1-a)^2 + b^2 + c^2 - 2bc - d^2 \le 1 -4\tau + 8\tau^2  + (b-c)^2 - d^2. \]
However, $\tau$ satisfies the equation $\tfrac{1}{2} + \tfrac{7}{2}\tau^2 = 1 -4\tau + 8\tau^2$ so the above simplifies to give 
\begin{equation}
\label{bmc}
  b - c \ge d.
\end{equation}
Note that since $b \ge d$ we have $1 \ge a + b+d \ge 3d$ and thus $d \le \frac{1}{3}$.  
At this point we have $c \ge \frac{-d + \sqrt{d^2 + 4\tau^2}}{2}$ and $b \ge \frac{d + \sqrt{d^2 + 4\tau^2}}{2}$  and we will show that this contradicts (\ref{111}).  To see this, note that under the assumption $a \ge b \ge c$ and $a+b+c = 1-d$ the quantity $a^2 + b^2 + c^2$ is maximized when $b$ and $c$ are as small as possible and $a$ is as large as possible.  Thus
\begin{align*}
 \tfrac{1}{2} + \tfrac{3}{2}\tau^2
 	&< a^2 + b^2 + c^2 + d(1-d)	\\
	&\le (1 - d - \sqrt{d^2 + 4\tau^2} )^2 + \tfrac{\left( d + \sqrt{d^2 + 4\tau^2} \right)^2}{4} +   \tfrac{\left( -d + \sqrt{d^2 + 4\tau^2} \right)^2}{4} + d(1-d)	\\
	&= 1 + 2d^2 - d + 6\tau^2 -2(1-d) \sqrt{d^2 + 4\tau^2}
\end{align*}
Using the identity $\tfrac{1}{2} + \tfrac{9}{2}\tau^2 = 4\tau$ and rearranging gives us
\begin{equation}
\label{lastineq}
  2(1-d) \sqrt{d^2 + 4\tau^2} < 4\tau + 2d^2 - d.
\end{equation}
The above equation immediately implies $d > 0$. 
From here a straightforward calculation gives the contradiction $d > \frac{1 - 2\tau^2 + \sqrt{ (1 - 2\tau^2)^2 +16(1 - 23\tau^2) } }{8} \doteq 0.485 > \frac{1}{3}$.  
(To check this by hand, square both sides of (\ref{lastineq}) and observe that the left and right sides are degree 4 polynomials in $d$ with matching highest order terms and matching constants; 
cancel these terms, divide by $d$, apply the quadratic formula and use the fact $9\tau^2-8\tau+1=0$ for simplification.) 
\end{proof}

\bibliographystyle{amsplain}

\begin{aicauthors}
\begin{authorinfo}[ra]
  Ron Aharoni\\
  Department of Mathematics, Technion -- Israel Institute of Technology,\\
  Technion City, Haifa 3200003, Israel\\
  ra\imageat{}technion\imagedot{}ac\imagedot{}il\\ 
  \url{https://raharoni.net.technion.ac.il/}
\end{authorinfo}
\begin{authorinfo}[md]
  Matt DeVos\\
  Deparment of Mathematics, Simon Fraser University\\
  Burnaby, Canada\\
  mdevos\imageat{}sfu\imagedot{}ca\\
  \url{http://www.sfu.ca/~mdevos/}
\end{authorinfo}
\begin{authorinfo}[sghm]
  Sebasti\'an Gonz\'alez Hermosillo de la Maza\\
  Deparment of Mathematics, Simon Fraser University\\
  Burnaby, Canada\\
  sga89\imageat{}sfu\imagedot{}ca
\end{authorinfo}
\begin{authorinfo}[am]
  Amanda Montejano\\
  UMDI Facultad de Ciencias, UNAM Juriquilla\\
  Quer\'{e}taro, Mexico\\
  amandamontejano\imageat{}ciencias\imagedot{}unam\imagedot{}mx\\
  \url{https://sites.google.com/site/amandamontejanohomepage/home}
\end{authorinfo}
\begin{authorinfo}[rs]
  Robert \v{S}\'amal\\
  Computer Science Institute, Charles University\\
  Prague, Czech Republic\\
  samal\imageat{}iuuk\imagedot{}mff\imagedot{}cuni\imagedot{}cz\\
  \url{https://iuuk.mff.cuni.cz/~samal/}
\end{authorinfo}
\end{aicauthors}

\end{document}